\documentclass[11pt]{article}
\usepackage{graphicx}
\usepackage{amsthm, amsmath, amssymb, tikz, bm}
\usepackage{mathtools}
\usepackage{xifthen}
\usepackage[colorlinks=true,
linkcolor=blue,citecolor=blue,
urlcolor=blue]{hyperref}

\usepackage[margin=1.2in]{geometry} 

\usepackage[shortlabels]{enumitem}
\usepackage{todonotes}
\usetikzlibrary{calc,shapes, backgrounds}
\allowdisplaybreaks

\newtheorem{lemma}{Lemma}
\newtheorem{theorem}{Theorem}

\newtheorem{definition}{Definition}

\newtheorem{remark}{Remark}

\newtheorem{proposition}{Proposition}

\newcommand{\ds}{\displaystyle}
\newcommand{\dss}{\displaystyle\sum}

\newcommand{\Var}{{\textrm{Var}}}
\newcommand{\E}{{\rm E}}
\newcommand{\lp}{\left(}
\newcommand{\rp}{\right)}

\newcommand{\cH}{\mathcal{H}}

\DeclarePairedDelimiter{\abs}{\lvert}{\rvert}%
\DeclarePairedDelimiter\sqbrac{\lbrack}{\rbrack}

\newcommand{\paren}[1]{\lp#1\rp}%

\newcommand{\Prob}[1]{\Pr\lp#1\rp}

\title{Concentration inequalities in spaces of random configurations with positive Ricci curvatures}

\author{
Linyuan Lu
\thanks{University of South Carolina, Columbia, SC 29208,
({\tt lu@math.sc.edu}). This author was supported in part by NSF
grant DMS-1600811.} \and
Zhiyu Wang \thanks{University of South Carolina, Columbia, SC 29208,
({\tt zhiyuw@math.sc.edu}). This author was supported in part by NSF
grant DMS-1600811.} 
}

\begin{document}
\maketitle
\begin{abstract}
In this paper, we prove an Azuma-Hoeffding-type inequality in several classical models of random configurations, including the Erd\H{o}s-R\'enyi random graph models $G(n,p)$ and $G(n,M)$, the random $d$-out(in)-regular directed graphs, and the space of random permutations.
The main idea is using Ollivier's work on the Ricci curvature of Markov chairs on metric spaces.
 Here we give a cleaner form of such concentration inequality in graphs. Namely, we show that for any Lipschitz function $f$ on any graph (equipped with an ergodic random walk and thus an invariant distribution $\nu$)
 with Ricci curvature at least $\kappa>0$, we have
 \[\nu \left( |f-E_{\nu}f| \geq t \right) \leq 2\exp\left( -\frac{t^2\kappa}{7}  \right).\]
\end{abstract}

\section{Introduction}


One of the main tools in probabilistic analysis and random graph theory is the concentration inequalities, which are meant to bound the probability that a random variable deviates from its expectation. Many of the classical concentration inequalities (such as those for binomial distributions) provide best possible deviation results with exponentially small probabilistic bounds. 
Such concentration inequalities usually require certain independence assumptions (e.g., the random variable is a sum of independent random variables). For concentration inequalities without the independence assumptions, one popular approach is the martingale method. A martingale is a sequence of random variables $X_0, X_1, \ldots, X_n$ with finite means such that $E[X_{i+1} \vert X_i, X_{i-1}, \ldots, X_0] = X_i$ for all $0\leq i < n$. For $\textbf{c} = (c_1, c_2, \ldots, c_n)$ with positive entries, a martingale $X$ is said to be $\textbf{c}$-Lipschitz if $|X_i - X_{i-1}| \leq c_i$ for $i\in [n]$. A powerful tool for controlling martingales is the Azuma-Hoeffding inequality \cite{Azuma, Hoeffding}: if a martingale is $\textbf{c}$-Lipschitz, then 
$$\Prob{  |X-E[X]| \geq t} \leq 2\exp{ \lp -\frac{t^2}{2\sum_{i=1}^n c_i^2} \rp}.$$
For more general versions of martingale inequalities as well as applications of martingale inequalities, we refer the readers to \cite{AlonSpencer, Chung-Lu}.

A graph $G = (V,E)$ is a pair of the vertex set $V$ and the edge set $E$ where each edge is an unordered pair of two vertices. Given a vertex $v\in V$, we use $\Gamma(v)$ to denote the set of \textit{open neighbors} of $v$ in $G$, i.e., $\Gamma(v) = \{u \in V: vu \in E\}$. Moreover, let $N(v) = \Gamma(v) \cup \{v\}$ be the \textit{closed neighbors} of $v$.
A graph parameter/function $X$ is called {\em vertex-Lipschitz} if $|X(G_1)-X(G_2)|\leq 1$ whenever $G_1$ and $G_2$ can be made isomorphic by deleting one vertex from each. A graph parameter $X$ is called \textit{edge-Lipschitz} if $|X(G_1) - X(G_2)| \leq 1$ whenver $G_1$ and $G_2$ differs by an edge. Many graph parameters are vertex(edge)-Lipschitz, e.g., the independence number $\alpha(G)$,
the chromatic number $\chi(G)$, the clique number $\omega(G)$, the domination number $\gamma(G)$, the matching number $\beta(G)$, etc. 

Concentration inequalities are among the most important tools in the probabilistic analysis of random graphs. The classical  binomial random graph model, denoted by $G(n,p)$, is a random graph model in which a graph with $n$ vertices is constructed by connecting the vertices randomly such that each vertex pair appears as an edge with probability $p$ independently from every other edge. The Erd\H{o}s-R\'enyi random graph model $G(n,M)$ is the model, in which a graph is chosen uniformly at random from the collection of all graphs with $n$ vertices and $m$ edges. A standard application of the Azuma-Hoeffding inequality gives us that for any vertex-Lipschitz function $X$ defined on a vertex-exposure martingale (see e.g. \cite{AlonSpencer} for definition), we have
\begin{equation}\label{eq:martinale}
\Pr(|X-\E(X)|\geq t)\leq 2\exp \paren{-\frac{t^2}{2n}}.
\end{equation}
Similar concentration results can be obtained for edge-exposure martingale as well.

In this paper, we will take an alternative approach for such an inequality. The main idea is using Ollivier's work \cite{Ollivier} on the Ricci curvature of Markov chairs on metric spaces. Although the Ricci curvature of graphs has been introduced  since 2009,
it has not been widely used by the communities of combinatorists and graph theorists. In this paper, we prove a clean concentration result (Theorem \ref{thm:main_tool}) on graphs with positive Ricci curvature. Then we show that it can be applied to some classical models of 
 random configurations including the Erd\H{o}s-R\'enyi random graph model $G(n,p)$ and $G(n,M)$, the random $d$-out(in)-regular directed graphs, and the space of random permutations, through a geometrization process.
 
 Consider a graph (loops allowed) $G=(V,E)$ equipped with a random work $m:=\{m_v\colon v\in V\}$.
 Here for each vertex $v$, $m_v\colon N(v)\to [0,1]$ is a distribution, i.e., $\sum_{x\in N(v)}m_v(x)=1$.
 Assume that this random walk is {\em ergodic} so that an invariant distribution $\nu$ exists. In the context of random walks on graphs, in order for the random walk to be ergodic, it is sufficient that the underlying graph $G$ is connected and non-bipartite. Note that $\nu$ is a probability measure on $V$. It turns $V$ into a probability space. A function $f\colon V\to \mathbb{R}$ is called $c$-Lipschitz on $G$ if 
 \begin{equation} \label{lipischtz}
  |f(u)-f(v)|\leq c \quad \mbox{ for any } uv\in E(G).     
 \end{equation}
 We have the following theorem on the concentration result of $f$. All we need is that the graph $G$ (equipped with a random walk) has positive Ricci curvature at least $\kappa>0$. (See the definition of Ricci curvature (in Ollivier's notion) in next section.) 
 
 \begin{theorem}\label{thm:main_tool}
Suppose that a graph $G=(V,E)$ equipped with an ergodic random walk $m$ (and invariant distribution $\nu$) has a positive Ricci curvature at least $\kappa>0$.
Then for any $1$-Lipschitz function $f$ and any $t\geq 1$,  we have 
\begin{align}\label{thm:conc0}
    \nu\lp f-E_{\nu}[f] > t\rp &\leq \exp{\lp\frac{-t^2\kappa}{7}\rp},\\
     \nu\lp f-E_{\nu}[f] < -t\rp &\leq \exp{\lp\frac{-t^2\kappa}{7}\rp}.
\end{align}
\end{theorem}

\begin{remark}
The constant $7$ can be improved to $5$ if 
$\kappa \to 0$ as $|V(G)| \to \infty$.
It can be improved to $1+o(1)$ if we further assume
$t\kappa \to 0$ as $|V(G)| \to \infty$.
\end{remark}

\begin{remark}
Ollivier \cite{Ollivier} proved a concentration inequality for any random walk on a metric space with positive Ricci curvature at least $\kappa>0$ and unique invariant distribution $\nu$. His result is more general but more technical to apply in the context of graphs. In particular, he defined two quantities related to the local behavior of the random walk: the {\em diffusion constant} $\sigma(x)$ and the {\em local dimension} $n_x$ at vertex $x$. Moreover, define $D^2_x = \frac{\sigma(x)^2}{n_x \kappa}$, $D^2 = E_{\nu}[D_x^2]$, 
$t_{max}=\frac{D^2}{\max(\sigma_\infty, 2C/3)}$ where $C$ satisfies that the function $x \to D_x^2$ is $C$-Lipschitz.
He proved (\cite{Ollivier} Theorem 33, on page 834) for any $1$-Lipschitz function $f$ and
for any $t\leq t_{max}$, we have
  \begin{equation}\label{conc1_Ollivier}
    \nu\lp f-E_{\nu}[f] > t\rp\leq \exp{\lp\frac{-t^2}{6D^2}\rp}.
\end{equation}
and for $t\geq t_{max}$,
  \begin{equation}\label{conc2_Ollivier}
    \nu\lp f-E_{\nu}[f] > t\rp\leq \exp{\lp\frac{-t^2}{6D^2}
    -\frac{t-t_{max}}{\max(3\sigma_\infty, 2C)}\rp}.
\end{equation}
\end{remark}

\begin{remark}
Note in Ollivier's result for graphs, we have $D^2=O(\kappa^{-1})$ and $\sigma_\infty\approx 1$. Inequality \eqref{thm:conc0} has about the same power as Inequalities \eqref{conc1_Ollivier} and \eqref{conc2_Ollivier},
but cleaner; thus is easier to apply in the context of graphs.
\end{remark}

Besides Ollivier's definition of Ricci curvature, another notion of Ricci curvature on discrete spaces, via geodesic convexity of the entropy (in the spirit of Sturm \cite{Sturm}, Lott and Villani \cite{LV}), was proposed in \cite{Mass} and systematically studied in \cite{EM} and \cite{Mielke}. Similar Gaussian-type concentration inequalities (as ones in Theorem \ref{thm:main_tool}) in this notion of Ricci curvature was proven in \cite{EM}. 
Erbar, Maas, and Tetali \cite{EMT} recently calculated the Ricci curvature lower bound of some classical random walks, e.g., the Bernoulli-Laplace model and the random  transposition  model of permutations. 

In this paper, we adopt Ollivier's notion of coarse Ricci curvature as it does not require the reversibility of the random walk on graphs.
The paper is organized as follows. In Section \ref{sec:ricci}, we will give the history and definitions of Ricci curvature. The proof of Theorem \ref{thm:main_tool} will be given in Section \ref{sec:main_proof}.  In last section, we will give applications of Theorem \ref{thm:main_tool} in four classical models of random configurations, including the Erd\H{o}s-R\'enyi random graph model $G(n,p)$ and $G(n,M)$, the random $d$-out(in)-regular directed graphs, and the space of random permutations.

\section{Ricci Curvatures of graphs}\label{sec:ricci}
In Riemannian geometry, spaces with positive Ricci curvature enjoy very nice properties, some of them with probabilistic interpretations. Many interesting properties are found on manifolds with non-negative Ricci curvature or on manifolds with Ricci curvature bounded below. 
The definition of the Ricci curvature on metric spaces first came from the Bakry and Emery notation \cite{Bakry-Emery} who defined the ``lower Ricci curvature bound" through the heat semigroup $(P_t)_{t\geq 0}$ on a metric measure space. Ollivier \cite{Ollivier} defined the coarse Ricci curvature of metric spaces in terms of how much small balls are closer (in Wasserstein transportation distance) then their centers are. This notion of coarse Ricci curvature on discrete spaces was also made explicit in the Ph.D. thesis of Sammer \cite{Sammer}.
Under the assumption of positive curvature in a metric space, Gaussian-like or Poisson-like concentration inequalities can be obtained. Such concentration inequalities have been investigated in \cite{Joulin} for time-continuous Markov jump processes and in \cite{Ollivier, JO} in metric spaces.

Graphs and manifolds share some similar properties through Laplace operators, heat kernels and random walks, etc. A series of work in this area were done by Chung, Yau and their coauthors \cite{Chung-G-Yau, Chung-Yau94, Chung-Yau95, Chung-Yau96, Chung-Yau97A, Chung-Yau97B, Chung-Yau99, Chung-Yau99B, Chung-Yau00, Chung-Yau00B, Chung-Yau00C}.
The first definition of Ricci curvature on graphs was introduced by Chung and Yau in \cite{Chung-Yau95}. For a more general definition of Ricci curvature, Lin and Yau \cite{LY} gave a generalization of lower Ricci curvature bound in the framework of graphs.  Lin, Lu, and Yau \cite{LLY} defined a new kind of Ricci curvature on graphs, which is based on Ollivier's work \cite{Ollivier}.

In this paper, we will use the same notation as in \cite{LLY}.
A probability distribution (over the vertex set $V(G)$) is a mapping $m: V\to [0,1]$ satisfying $\sum_{x\in V} m(x) = 1$. Suppose two probability distributions $m_1$ and $m_2$ have finite support. A coupling between $m_1$ and $m_2$ is a mapping $A: V\times V \to [0,1]$ with finite support so that 
$$\dss_{y \in V} A(x,y) = m_1(x) \textrm{ and } \dss_{x\in V} A(x,y) = m_2(y).$$
Let $d(x,y)$ be the graph distance between two vertices $x$ and $y$. The \textit{transportation distance} between two probability distributions $m_1$ and $m_2$ is defined as follows:
$$W(m_1, m_2) = \inf_A \dss_{x,y\in V} A(x,y) d(x,y).$$
where the infimum is taken over all coupling $A$ between $m_1$ and $m_2$. By the duality theorem of a linear optimization problem, the transportation distance can also be written as follows:
$$W(m_1, m_2) = \sup_f \dss_{x\in V} f(x) \lp m_1(x)-m_2(x)\rp$$
where the supremum is taken over all $1$-Lipschitz functions $f$.

A random walk $m$ on $G=(V,E)$ is defined as a family of probability measures $\{m_v(\cdot)\}_{v\in V}$ such that $m_v(u) = 0$ for all $\{v,u\} \notin E$. It follows that  $m_v(u) \geq 0$ for all $v,u\in V$ and $\sum_{u\in N(v)} m_v(u) = 1$. The Ricci cuvature $\kappa$ of $G$ can then be defined as follows:

\begin{definition}
Given $G=(V,E)$, a random walk $m = \{m_v(\cdot)\}_{v\in V}$ on $G$ and two vertices $x,y\in V$, 
$$\kappa(x,y) = 1 - \frac{W(m_x, m_y)}{d(x,y)}.$$
\end{definition}
\begin{remark}
We say a graph $G$ equipped with a random walk $m$ has Ricci curvature at least $\kappa_0$ if $\kappa(x,y) \geq \kappa_0$ for all $x,y \in V$.
\end{remark}
For $0\leq \alpha < 1$, the {\em $\alpha$-lazy random walk} $m_x^{\alpha}$ (for any vertex $x$), is defined as 
\[
m_x^{\alpha}(v) = \begin{cases} 
                        \alpha & \textrm{ if $v=x$,}\\
                        (1-\alpha)/d_x &\textrm{ if $v\in \Gamma(x)$,}\\
                        0 & \textrm{ otherwise.}
                    \end{cases}
\]
In \cite{LLY}, Lin, Lu and Yao defined the Ricci curvature of graphs based on the $\alpha$-lazy random walk as $\alpha$ goes to $1$. More precisely,
for any $x,y \in V$, they defined the $\alpha$-Ricci-curvature $\kappa_{\alpha}(x,y)$ to be 
$$\kappa_{\alpha}(x,y) = 1 - \frac{W(m_x^{\alpha}, m_y^{\alpha})}{d(x,y)}$$ and the Ricci curvaure $\kappa_{\textrm{LLY}}$ of $G$ to be 
\[\kappa_{\textrm{LLY}}(x,y) = \ds\lim_{\alpha \to 1} \frac{\kappa_{\alpha}(x,y)}{(1-\alpha)}.\]
They showed \cite{LLY} that $\kappa_{\alpha}$ is concave in $\alpha \in [0,1]$ for any two vertices $x,y$. Moreover, 
\[\kappa_{\alpha}(x,y) \leq (1-\alpha) \frac{2}{d(x,y)}.\]
for any $\alpha \in [0,1]$ and any two vertices $x$ and $y$.

In the context of graphs, the following lemma shows that it is enough to consider only $\kappa(x,y)$ for $xy\in E(G)$. 
\begin{lemma}\cite{Ollivier, LLY}\label{lem:kappa_neighbor}
If $\kappa(x,y)\geq \kappa_0$ for any edge $xy\in E(G)$, then $\kappa(x,y)\geq \kappa_0$ for any pair of vertices $(x,y)$.
\end{lemma}

\section{Proof of Theorem \ref{thm:main_tool}}\label{sec:main_proof}

We first define an averaging operator associated to the random walk.

\begin{definition}[Discrete averaging operator]
Given a function $f: X \to \mathbb{R}$, let the averaging operator $M$ be defined as 
\[Mf(x) := \dss_{y \in V} f(y) \cdot m_x(y).\]
\end{definition}
The following proposition shows a Lipschitz contraction property in the metric measure space. We include its proof here for the sake of completeness.

\begin{proposition}[Lipschitz contraction]\cite{Ollivier, DGW}\label{prop:Lip-contraction}
Let $(G, d, m)$ be a random walk on a simple graph $G$. Let $\kappa \in \mathbb{R}$. Then the Ricci curvature of $G$ is at least $\kappa$, if and only if, for every $k$-Lipschitz function $f: X\to \mathbb{R}$, the function $Mf$ is $k(1-\kappa)$-Lipschitz.
\end{proposition}
\begin{proof}
Suppose that the Ricci curvature of $G$ is at least $\kappa$. For $x ,y\in V$, let $A: V\times V \to [0,1]$ be the optimal coupling measure of $m_x$ and $m_y$.
\begin{align*}
    Mf(y)-Mf(x) &= \dss_{u\in V} f(u) m_y(u) - \dss_{u\in V} f(u)m_x(u)\\
                &= \dss_{u\in V} f(u) \dss_{v\in V} A(v,u) - \dss_{u\in V} f(u) \dss_{v\in V} A(u,v)\\
                &= \dss_{u,v} \lp f(v)-f(u)\rp A(u,v)\\
                &\leq k\dss_{u,v} d(u,v) A(u,v)\\
                & = kW(m_x, m_y)\\
                &= k(1-\kappa(x,y))d(x,y)
\end{align*}
Conversely, suppose that whenever $f$ is $1$-Lipschitz, $Mf$ is $(1-\kappa)$-Lipschitz. Then by the duality theorem for the transportation distance, we have that for all $x,y \in V(G)$,
\begin{align*}
    W(m_x, m_y) &= \ds\sup_{\textrm{$f$ $1$-Lipschitz}} \dss_{z\in V} f(z) \lp m_x(z)-m_y(z)\rp \\
                &= \ds\sup_{\textrm{$f$ $1$-Lipschitz}} M f(x) -Mf(y)\\
                &\leq (1-\kappa) d(x,y).
\end{align*}
It follows that 
$$\kappa(x,y) = 1 - \frac{W(m_x, m_y)}{d(x,y)} \geq \kappa.$$
\end{proof}


\begin{remark}\label{rmk:var_bound}
Note that for any constant $c$, 
\begin{equation}
    \Var(f) = E\sqbrac{(f-c)^2} - \paren{E[f]-c}^2.
\end{equation}

\noindent Thus for any $x\in V$ and an $\alpha$-Lipschitz function $f: \textrm{Supp } m_x \to \mathbb{R}$,  
\begin{align*}
  \Var_{m_x}f &\leq  E_{m_x}\sqbrac{(f-f(x))^2}\\
              &\leq \dss_{y\in \textrm{Supp $m_x$}} (f(y)-f(x))^2 m_x(y)\\
              &\leq \alpha^2.
\end{align*}
\end{remark}

\begin{lemma}\cite{LLY, Ollivier}\label{lem:kappa_diam}
Let $G$ be a finite graph with Ricci curvature at least $\kappa > 0$. Then 
$$\kappa \leq \frac{2}{\textrm{diam}(G)}.$$
Moreover, if $m_x(x) = \alpha$ for all $x\in V(G)$, then 
$\kappa \leq (1-\alpha)\frac{2}{\textrm{diam}(G)}$.
\end{lemma}

The following lemma is similar to Lemma 38 in \cite{Ollivier}.

\begin{lemma}\label{lem:average-operator}
Let $\phi:V(G) \to \mathbb{R}$ be an $\alpha$-Lipschitz function with $\alpha \leq 1$. Then for $x\in V(G)$, we have 
\[\paren{Me^{\lambda \phi}}(x) \leq e^{\lambda M \phi(x)+ \frac{1}{2}\lambda^2 e^{2\lambda} \alpha^2}.\]
\end{lemma}
\begin{proof}
For any smooth function $g$ and any real-valued random variable $Y$, a Taylor expansion with Lagrange remainder gives 
\[Eg(Y)\leq g(EY) + \frac{1}{2}(\textrm{sup }g'')\Var Y.\]
Applying this with $g(Y) = e^{\lambda Y}$, we get 
\[(Me^{\lambda \phi})(x) = E_{m_x} e^{\lambda \phi} \leq e^{\lambda M \phi(x)}+ \frac{\lambda^2}{2} \paren{\ds\sup_{\textrm{Supp }m_x} e^{\lambda \phi}} \Var_{m_x} \phi.\]
Note that diam Supp $m_x \leq 2$ and $\phi$ is $\alpha$-Lipschitz, it follows that   
\[\displaystyle\sup_{\textrm{Supp }m_x} \phi \leq E_{m_x} \phi + \alpha\cdot  (\textrm{diam Supp } m_x) \leq E_{m_x} \phi + 2\alpha .\]
Moreover, by Remark \ref{rmk:var_bound}, $\Var_{m_x}\phi \leq \alpha^2$. Hence we have that
\begin{align*}
    \paren{Me^{\lambda \phi}}(x) &\leq e^{\lambda M \phi(x)} + \frac{\lambda^2}{2} (\alpha^2) e^{\lambda M\phi(x) + 2\lambda\alpha} \\
            &\leq e^{\lambda M \phi(x)} \lp  1 + \frac{\lambda^2}{2} \alpha^2 e^{2\lambda\alpha}\rp\\
            &\leq \exp\lp\lambda M\phi(x) + \frac{1}{2}\lambda^2\alpha^2 e^{2\lambda \alpha}\rp.
\end{align*}
\end{proof}

\begin{proof}[Proof of Theorem \ref{thm:main_tool}]

First, note that since $f$ is $1$-Lipschitz, it follows that $\abs*{f(x)-f(y)}\leq diam(G)$ for any $x,y\in V(G)$. Hence if $t > \frac{2}{\kappa}$, then 

\[\Prob{\abs*{f-E_{\nu}[f]} \geq t} \leq \Prob{\abs*{f-E_{\nu}[f]} > \frac{2}{\kappa}} \leq \Prob{\textrm{diam}(G) > \frac{2}{\kappa}}=0,\]
in which case we are done. So from now on, assume $t \leq 2/\kappa$. 

Apply Lemma \ref{lem:average-operator} iteratively and use Proposition \ref{prop:Lip-contraction}, we obtain that for any $i\geq 1$,
\begin{align*}
    M^i(e^{\lambda f}) &\leq e^{\lambda M^i f}\cdot \ds\prod_{j=0}^{i-1} \exp{\lp\frac{1}{2}\lambda^2 (1-\kappa)^{2j} e^{2\lambda}\rp}\\
    &\leq \exp{\lp\lambda M^i f + \frac{1}{2}\lambda^2 e^{2\lambda}\dss_{j=0}^{i-1} (1-\kappa)^{2j}\rp.}
\end{align*}
Meanwhile, $(M^i e^{\lambda f})(x)$ tends to $E_{\nu} e^{\lambda f}$.
Hence 
\begin{align*}
    E_{\nu} e^{\lambda f} &\leq \lim_{i\to \infty}  \exp{\lp\lambda M^i f + \frac{1}{2}\lambda^2 e^{2\lambda}\dss_{j=0}^{i-1} (1-\kappa)^{2j}\rp}\\
                &\leq  \exp{\lp \lambda E_{\nu} f + \frac{\lambda^2 e^{2\lambda}}{2\kappa(2-\kappa)} \rp.}
\end{align*}

Let $\lambda_0$ be the root of the equation 
$x \cdot e^{2x} = 2(2-\kappa)$ and set $\lambda = \frac{t \kappa \lambda_0}{2}$.
Note that since $t \leq \frac{2}{\kappa}$, we have $\lambda \leq \lambda_0$. Now, we have
\begin{align}
    \Prob{f-E_{\nu}f \geq t} &\leq \Prob{e^{\lambda f} \geq e^{t\lambda + \lambda E_{\nu}f}} \nonumber\\
            &\leq E_{\nu} e^{\lambda f}\cdot e^{-t \lambda -\lambda E_{\nu} f} \nonumber\\
            &\leq \exp{\lp  -t \lambda + \frac{\lambda^2 e^{2\lambda}}{2\kappa(2-\kappa)}\rp} \nonumber\\
            &\leq \exp{\lp  -t\lambda + \frac{\lambda t \lambda_0 e^{2\lambda}}{4(2-\kappa)}\rp} \label{eq:optim_const}\\
            &\leq \exp{\lp  -t\lambda + \frac{\lambda t \lambda_0 e^{2\lambda_0}}{4(2-\kappa)}\rp} \nonumber\\
            &= \exp{\lp  -\frac{1}{2}t \lambda\rp} \nonumber\\
            &\leq \exp{\lp -\frac{t^2\kappa \lambda_0}{4}\rp} \nonumber
\end{align}
where $\lambda_0$ is the solution to $x \cdot e^{2x} = 2(2-\kappa)$. 
If $G$ is the complete graph, then $|f-E_\nu(f)|\leq 1$ holds for all vertices.
Inequality \ref{thm:conc0} holds. If $G$ is not the complete graph, then
we must have $\kappa\leq 1$ (otherwise, contradiction to $diam(G)\leq \frac{2}{\kappa}$). Thus $\lambda_0\leq 0.60108...$, which is the root of
$x \cdot e^{2x}=2.$ We have $\frac{\lambda_0}{4}>\frac{1}{7}.$
Hence we obtain that
\[\Prob{f-E_{\nu}f \geq t} \leq \exp\lp -\frac{t^2\kappa}{7}  \rp.\]
If $\kappa \to 0$ as $|V(G)| \to \infty$ (which is true in all the examples in Section \ref{sec:app}), then we have $\lambda_0\to 0.80290...$ which is the root of
$x \cdot e^{2x}=4.$ We have $\frac{\lambda_0}{4}>\frac{1}{5}.$ We have
\[\Prob{f-E_{\nu}f \geq t} \leq \exp\lp -\frac{t^2\kappa}{5} \rp.\]
Furthermore, if $\kappa \to 0$ and $t\kappa \to 0$ as $|V(G)| \to \infty$, then continuing from inequality \eqref{eq:optim_const}, we have that $e^{2\lambda} \to 1$ and $(2-\kappa) \to 2$ (as $|V(G)| \to \infty$). By setting $\lambda_0 = 4$, we have
 \begin{align*}
     \Prob{f-E_{\nu}f \geq t} & \leq \exp{\lp  -t\lambda + \frac{\lambda t \lambda_0 e^{2\lambda}}{4(2-\kappa)}\rp} \\
                             &\leq \exp{\lp-\paren{\frac{1}{2}+o(1)}t\lambda\rp}\\
                             &\leq \exp{\lp -\paren{\frac{1}{4}+o(1)}t^2\kappa \lambda_0 \rp} \\
                             &\leq \exp{\lp (1+o(1)) t^2\kappa\rp.}
 \end{align*}
The lower tail can be obtained from the upper tail by changing $f$ to $-f$ since $-f$ is also $1$-Lipschitz.
\end{proof}

\section{Applications to random models of configurations}\label{sec:app}

In order to apply Theorem \ref{thm:main_tool} to a finite probability space $(\Omega, \mu)$, we will construct a graph $H$ with the vertex set $\Omega$ such that $\mu$ is the invariant distribution over a proper random walk $m$ on $H$. 
We call the pair $(H, m)$ a {\em geometrization} of $(\Omega, \mu)$. In this section, we will give
geometrization of four popular random model of configuarations.

\subsection{Vertex-Lipschitz functions on \texorpdfstring{$G(n,p)$}{G(n,p)}}
Let $H$ be the graph such that $V(H)$ is the set of all labeled graphs with $n$ vertices. Moreover, two graphs $G_1, G_2 \in V(H)$ are adjacent in $H$ if and only if there exists some $v$ such that $G_1 - v = G_2 - v$.
Now define a random walk $m$ on $H$ as follows: Let $G \in V(H)$. Define

$$m_G(G') = \begin{cases}
    \frac{1}{n} \dss_{\substack{v\in V(G)\\G-v = G'-v}}p^{d_{G'}(v)}(1-p)^{n-1-d_{G'}(v)} & \textrm{ if $G'\in N_H(G)$,} \\
    0      &  \textrm{ otherwise.}
\end{cases}$$

\begin{proposition}
Let $\nu$ be the unique invariant distribution of the random walk defined above. A random graph $G$ picked according to $\nu$, satisfies that $\nu(G) = p^{e(G)}(1-p)^{\binom{n}{2}-e(G)}$.
\end{proposition}
\begin{proof} Observe that $H$ is not bipartite thus the random walk is ergodic.
It suffices to show that the distribution $\nu'(G) = p^{e(G)}(1-p)^{\binom{n}{2}-e(G)}$ for every $G$ is an invariant distribution for the random walk. Indeed, for every fixed $G \in V(H)$,
\begin{align*}
     &  \dss_{G'\in H} \nu'(G') m_{G'}(G) \\ 
  =  &  \dss_{v\in V} \dss_{G'-v= G-v} \nu'(G')  \frac{1}{n} p^{d_{G}(v)}(1-p)^{n-1-d_{G}(v)} \\
    =&  \dss_{v\in V} \frac{1}{n}p^{d_G(v)} (1-p)^{n-1-d_G(v)} \dss_{G'-v= G-v} \nu'(G')\\
   = &  \dss_{v\in V} \frac{1}{n}p^{d_G(v)} (1- p)^{n-1-d_G(v)} \cdot\\
    & \lp p^{e(G)-d_G(v)} (1-p)^{\binom{n-1}{2}-(e(G)-d_G(v))} \dss_{i=0}^{n-1} \binom{n-1}{i} p^{i} (1-p)^{n-1-i}   \rp\\
  = & \frac{1}{n}p^{e(G)} (1- p)^{\binom{n}{2}-e(G)} \dss_{v\in V}  \dss_{i=0}^{n-1} \binom{n-1}{i} p^{i} (1-p)^{n-1-i}   \\
  = & p^{e(G)} (1- p)^{\binom{n}{2}-e(G)}\\
  = & \nu'(G).
\end{align*}

\end{proof}

\begin{lemma}
Let $H$ and the random walk $m$ be defined as above. Then $$\kappa(G_1, G_2) \geq \frac{1}{n}$$ for all $G_1, G_2 \in V(H)$.
\begin{proof}
Again, by Lemma \ref{lem:kappa_neighbor}, we can assume that $G_1, G_2$ are neighbors in $H$. It then follows from definition that
\[\kappa(G_1, G_2) = 1- W(m_{G_1},m_{G_2}).\] 

Assume that $v$ is the unique vertex such that $G_1-v = G_2 -v$. When $G_1$ and $G_2$ differ by an edge, it is possible that there are two vertices $v$ satisfying $G_1-v = G_2 -v$. We remark that the analysis is similar.
Consider the support of $m_{G_1}$. For each $G_1' \in \Gamma(G_1)\backslash \{G_2\}$, we will match $G_1'$ with a distinct graph $\phi(G_1') \in N(G_2)$. There are two possible cases:
\begin{description}
\item Case $1$: $G_1 - v = G_1' -v$. Then it follows that $G_1'-v = G_2 -v$ and we let $\phi(G_1') = G_1'$.
\item Case $2$: $G_1 -u = G_1'- u$ for some $u\neq v$. In this case, we claim that for each $G_1'$ such that $G_1 -u = G_1' -u$, there exists a unique $G_2' = \phi(G_1')$ such that $G_2' - u = G_2 - u$ and $G_1' - v = G_2' - v$.
Indeed, let $G_2'$ be obtained from $G_2$ by replacing the neighbors of $u$ in $G_2$ by the neighbors of $u$ in $G_1'$. It's not hard to see that $G_2'-u = G_2 -u$ and $G_1' - v = G_2'-v$.
\end{description}

Let us now define a coupling $A$ (not necessarily optimal) between $m_{G_1}$ and $m_{G_2}$. Define $A: V(H) \times V(H) \to \mathbb{R}$ as follows:
\begin{equation}
    A(G_1', G_2') = \begin{cases}
     \frac{1}{n} \dss_{u\neq v} p^{d_{G_1}(u)}(1-p)^{n-1-d_{G_1}(u)} & \textrm{ if } G_2' = G_2, G_1' = G_1,\\
      m_{G_1}(G_1')& \textrm{ if } G_1' \in \Gamma(G_1)\backslash\{G_2\} \textrm{ and $G_2' = \phi(G_1')$,}\\
        0    &\textrm{ otherwise. }
    \end{cases}
\end{equation}
It follows that 
\begin{align*}
W(m_{G_1},m_{G_2}) &\leq \dss_{G_1',G_2'} A(G_1', G_2') d(G_1', G_2') \\
&\leq  \frac{1}{n} \dss_{u\neq v}\dss_{G'-u= G_1-u}\frac{1}{n} p^{d_{G'}(u)}(1-p)^{n-1-d_{G'}(u)}\\
&\leq  \frac{1}{n} \dss_{u\neq v} \dss_{i=0}^{n-1} \binom{n-1}{i} p^i (1-p)^{n-1-i} \\
&\leq \frac{n-1}{n}.
\end{align*}
Thus 
$$\kappa(G_1, G_2) \geq 1 - W(m_{G_1}, m_{G_2}) \geq \frac{1}{n}.$$

\end{proof}
\end{lemma}

It follows by Theorem \ref{thm:main_tool} that for any vertex-Lipschitz function $f$ on graphs, we have that 
$$\Prob{\abs*{f-E[f]} \geq t} \leq 2\exp \paren{-\frac{t^2}{5n}},$$
which in this context has the same strength as the Azuma-Hoeffding inequality on vertex-exposure martingale.

\subsection{Edge-Lipschitz functions on \texorpdfstring{$G(n,M)$}{G(n,M)}}\label{sec: randomMedges}
Let $G \sim G(n,M)$ be a random graph with $n$ vertices and $M$ edges. Let $H$ be the graph such that $V(H)$ is the set of all labeled graphs with $n$ vertices and $M$ edges. Moreover, two graphs $G_1, G_2 \in V(H)$ are adjacent in $H$ if and only if there exist two distinct vertex pairs $e_1, e_2$ such that $e_1 \in E(G_1)\backslash E(G_2)$, $e_2 \in E(G_2)\backslash E(G_1)$ and $G_1 - e_1  = G_2 - e_2$. In other words, $G_1, G_2$ are adjacent in $H$ if one can be obtained from the other by swapping an edge with a non-edge. It is easy to see that $H$ is a connected regular graph. Moreover, for every $G \in V(H)$, $d_H(G) = M \paren{\binom{n}{2}-M}$.

The following proposition is clear from the definition of $H$.
\begin{proposition}
If $G_1, G_2$ are adjacent in $H$, then there exists a unique pair of distinct vertex pairs $e_1, e_2$ such that $e_1 \in E(G_1)\backslash E(G_2)$, $e_2 \in E(G_2)\backslash E(G_1)$ and $G_1 - e_1  = G_2 - e_2$.
\end{proposition}

\noindent Now define a random walk $m$ on $H$ as follows:
Let $G \in V(H)$. Define 
$$m_G(G') = \begin{cases}
 \frac{1}{M \paren{\binom{n}{2}-M}+1}  & \textrm{ if $G'\in N_H(G)$,} \\
    0      &  \textrm{ otherwise.}
\end{cases}$$
It's easy to see that for a fixed $G$, $\sum_{G'} m_G(G') = 1$.

\begin{proposition}
Let $\nu$ be the unique invariant distribution of the random walk defined above. A random graph $G$ picked according to $\nu$, is equally likely to be one of the $\binom{\binom{n}{2}}{M}$ graphs that have $M$ edges.
\end{proposition}
\begin{proof} Observe that $H$ is not bipartite thus the random walk is ergodic.
It suffices to show that $\nu'(G) = \binom{\binom{n}{2}}{M}^{-1}$ for every $G$ is an invariant distribution for the random walk. Indeed, for every fixed $G \in V(H)$,
\begin{align*}
    \dss_{G'\in H} \nu'(G') m_{G'}(G) 
                &= \binom{\binom{n}{2}}{M}^{-1} \dss_{G' \in N(G)} m_{G'}(G)\\
                &= \binom{\binom{n}{2}}{M}^{-1} \dss_{G'\in N(G)} m_{G}(G')\\
                &= \binom{\binom{n}{2}}{M}^{-1}\\
                &= \nu'(G).
\end{align*}
Since $\nu$ is the unique invariant distribution, it follows then that $\nu= \nu'$. 
\end{proof}

\begin{lemma}
Let $H$ and the random walk $m$ be defined as above. Then $$\kappa(G_1, G_2) \geq \frac{\binom{n}{2}}{M\paren{\binom{n}{2}-M}+1}$$ for all $G_1, G_2 \in H$.
\begin{proof}
By Lemma \ref{lem:kappa_neighbor}, we can assume that $G_1, G_2$ are neighbors in $H$. It then follows from definition that
\[\kappa(G_1, G_2) = 1- W(m_{G_1},m_{G_2}).\] 

Suppose $e_1, e_2$ are the unique vertex pairs with $e_1 \in E(G_1), e_2 \notin E(G_1)$ such that $G_2 = G_1 - e_1 + e_2$. 
Consider the support of $m_{G_1}$, i.e., $N(G_1)$. For each $G_1' \in N(G_1)$, we will match $G_1'$ with a distinct graph $\phi(G_1') \in N(G_2)$. First, let $\phi(G_1) = G_1$ and $\phi(G_2) = G_2$.
For other neighbors $G_1' \in N(G_1)$, there are three types:
\begin{description}
\item Type $1$: $G_1 - e_1 = G_1' -e_3$ for some $e_3\neq e_2$. Then it follows that $G_1'-e_3 = G_2 -e_2$ and we let $\phi(G_1') = G_1'$.
\item Type $2$: $G_1 - e_3 = G_1' - e_2$ for some $e_3\neq e_1$. Then it follows that $G_1'-e_1 = G_2 - e_3$ and we let $\phi(G_1') = G_1'$.
\item Type $3$: $G_1 -e_3 = G_1'- e_4$ for some $e_3, e_4 \notin \{e_1, e_2\}$. In this case, we claim that there exists a unique $G_2' = \phi(G_1') \in N(G_2)$ such that $G_1'-e_1 = G_2'-e_2$.
Indeed, $G_2' = G_2 -e_3 + e_4$ will satisfy the aforementioned property.
\end{description}

Let us now define a coupling $A$ (not necessarily optimal) between $m_{G_1}$ and $m_{G_2}$. Define $A: V(H) \times V(H) \to \mathbb{R}$ as follows:
\begin{equation}
    A(G_1', G_2') = \begin{cases}
      \frac{1}{M \paren{\binom{n}{2}-M}+1} & \textrm{ if } G_1' \in N(G_1) \textrm{ and $G_2' = \phi(G_1')$,}\\
        0    &\textrm{ otherwise. }
    \end{cases}
\end{equation}

Let us verify that $A$ is a coupling of $m_{G_1}$ and $m_{G_2}$. Indeed, for each fixed $G_1'$, if $G_1'= G_1$, then $\sum_{G_2'} A(G_1', G_2') =A(G_1, G_1) = m_{G_1}(G_1)$; if $G_1'\neq G_1$, then $\sum_{G_2'} A(G_1', G_2') = A(G_1',\phi(G_1')) = m_{G_1}(G_1')$. Similarly, $\sum_{G_1'}A(G_1',G_2') = m_{G_2}(G_2')$.
Now by definition,
\begin{align*}
    W(m_{G_1},m_{G_2}) &\leq \dss_{G_1',G_2'} A(G_1', G_2') d(G_1', G_2') \\
    &\leq \dss_{G_1' \in N(G_1)}  A(G_1', \phi(G_1')) d(G_1', \phi(G_1'))\\
    &= \dss_{\substack{G_1' \in N(G_1)\\G_1' \textrm{is Type 3}}} A(G_1', \phi(G_1')) \\
    &\leq \paren{(M-1) \paren{\binom{n}{2}-M-1}} \cdot \frac{1}{M \paren{\binom{n}{2}-M}+1}.
\end{align*}
It follows that 
\begin{align*}
    \kappa(G_1, G_2) &= 1- W(m_{G_1},m_{G_2})\\
                     &\geq \frac{\binom{n}{2}}{M \paren{\binom{n}{2}-M}+1}.
\end{align*}
\end{proof}
\end{lemma}

Let $G(n,M)$ be an  Erd\H{o}s-R\'enyi random graph with $M$ edges. Let $F$ be a fixed graph and $X_{F}$ be the number of copies of $F$ in the random graph $G(n,M)$.  Denote the number of vertices and edges of $F$ by $v(F)$ and $e(F)$ respectively. Let $p = M/\binom{n}{2}$ and \textrm{Aut($F$)} denote the set of automorphisms of $F$. Then 
$$E[X_F] = (1+o(1))\frac{v(F)!}{|\textrm{Aut($F$)}|} \binom{n}{v(F)} p^{e(F)} = \Theta \lp n^{v(F)} p^{e(F)}\rp.$$
For a series of results on the upper tail of $X_F$ using different techniques, we refer the readers to the survey \cite{JR} and the paper \cite{JOR, Chatterjee, DK,DK2, AW}. For $G(n,M)$ in particular, Janson, Oleszkiewicz, Ruci\'nski \cite{JOR} showed the following theorem:

\begin{theorem}\cite{JOR}\label{thm:JOR-Gnm}
For every graph $F$ and for every $t> 1$, there exist constants $c(t,F) >0$ such that for all $n\geq v(F)$ and $e(F)\leq M\leq\binom{n}{2}$, with p:= $M/\binom{n}{2}$, 
$$\Prob{X_F \geq tE[X_F]} \leq \exp{\paren{-c(t,F) M^*_F(n,p)}},$$
where $M^*_F(n,p) \leq n^2p = O(M), M^*_{C_k}(n,p) = \Theta(n^2p^2)$ and $M^*_{K_k}(n,p) = \Theta(n^2 p^{k-1})$.
\end{theorem}

Let us now apply Theorem \ref{thm:main_tool} to obtain the concentration results from the perspective of the Ricci curvature. Recall that $H$ is defined as the graph such that $V(H)$ is the set of all labeled graphs with $n$ vertices and $M$ edges. Moreover, two graphs $G_1, G_2 \in V(H)$ are adjacent in $H$ if and only if there exist two distinct vertex pairs $e_1, e_2$ such that $e_1 \in E(G_1)\backslash E(G_2)$, $e_2 \in E(G_2)\backslash E(G_1)$ such that $G_1 - e_1  = G_2 - e_2$.

Again let $X_F$ be the random variable denoting the number of copies of $F$ in $G(n,M)$. For ease of reference, let $k = v(F)$.
Observe that $X_{F}$ is $\binom{n}{k-2}$-Lipschitz on $H$, i.e., if $G_1, G_2$ are adjacent in $H$, then $|X_F(G_1)-X_F(G_2)| \leq \binom{n}{k-2}$. Thus by Theorem \ref{thm:main_tool}, 

$$ \Prob{\frac{X_{F}}{\binom{n}{k-2}} > \frac{E[X_{F}]}{\binom{n}{k-2}} + \frac{t}{\binom{n}{k-2}}} \leq \exp{\lp -\frac{t^2\kappa}{5\binom{n}{k-2}^2}\rp}.$$ 
It follows that 
$$ \Prob{X_{F} > E[X_{F}] + t} \leq \exp{\lp -\frac{t^2\kappa}{5\binom{n}{k-2}^2}\rp}.$$ 
Let $p = M/\binom{n}{2}$. We then obtain that
\begin{equation}\label{eq:subgraph_curv}
    \Prob{X_F \geq tE[X_F]} 
                \leq \exp{\lp -\frac{\paren{(t-1)E[X_F]}^2\kappa}{5\binom{n}{k-2}^2}\rp}
                \leq \exp \paren{-C_k(t-1)^2 n^2 p^{2e(F)-1}}.      
\end{equation}
Note that when $p = \Theta(1)$, i.e., $M = \Theta \paren{ \binom{n}{2}}$, the concentration inequalities obtained from Theorem \ref{thm:main_tool} has the same asymptotic exponent as Theorem \ref{thm:JOR-Gnm}.
For other ranges of $p$ with $n^2 p \to \infty$, the asymptotic exponent in \eqref{eq:subgraph_curv} is worse than the bound in Theorem \ref{thm:JOR-Gnm}. Nonetheless, let us compare the bounds obtained from the Ricci curvature method with those obtained from other concentration inequalities. Janson and Ruci\'nski \cite{JR} surveyed the existing techniques on estimating the exponents for upper tails in the small subgraphs problem in $G(n,p)$ (ignoring logarithmic factors). Please see Figure \ref{f:subgraph_prob} for the summary.
\begin{figure}[htb]
	\begin{center}
    \includegraphics[width=\textwidth]{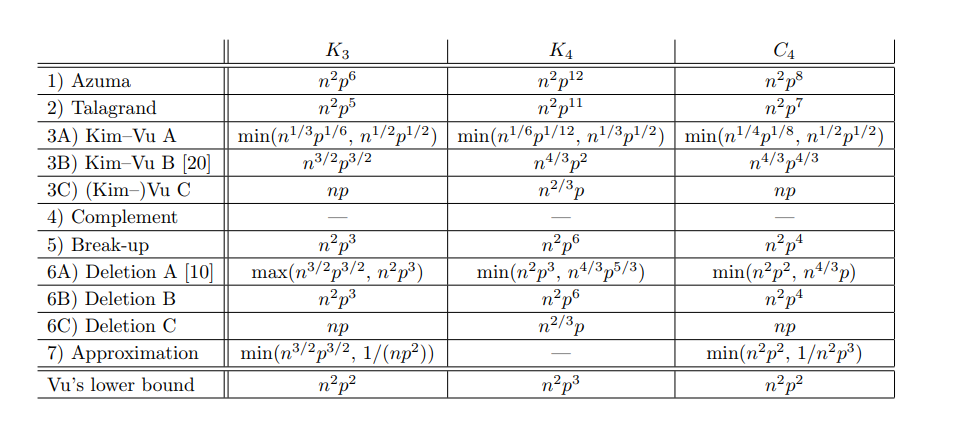}
    \caption{\cite{JR} Exponents for upper tails in the small subgraphs problem}
		\label{f:subgraph_prob}
    \end{center}
\end{figure}

Although we are mainly dealing with $G(n,M)$ in this section, it is well known that $G(n,M)$ and $G(n,p)$ with $p = M/\binom{n}{2}$ behaves similarly when $n^2p \to \infty$. 
Applying the inequalities in (\ref{eq:subgraph_curv}) to $K_3, K_4, C_4$ respectively, we have that the exponents (ignoring constant) obtained from the Ricci curvature method are $n^2 p^{5}$, $n^2p^{11}$ and $n^2p^7$ respectively. In this context, the concentration we obtained from Theorem \ref{thm:main_tool} has the same strength as Talagrand inequality and slightly stronger than Azuma's inequality.




\subsection{Edge-Lipschitz functions on random hypergraphs}
Let $\cH \sim \cH^k(n,M)$ be a random $k$-uniform hypergraph with $n$ vertices and $M$ edges. Let $H$ be a graph such that $V(H)$ is the set of all labeled $k$-uniform hypergraphs with $n$ vertices and $M$ edges. Moreover, two hypergraphs $\cH_1, \cH_2 \in V(H)$ are adjacent in $H$ if and only if there exist two distinct $k$-sets $h_1, h_2$ such that $h_1 \in E(\cH_1)\backslash E(\cH_2)$, $h_2 \in E(\cH_2)\backslash E(\cH_1)$ and $\cH_1 - h_1  = \cH_2 - h_2$. In other words, $\cH_1, \cH_2$ are adjacent in $H$ if one can be obtained from the other by swapping a hyperedge with a non-hyperedge. It is easy to see that $H$ is a connected regular graph. Moreover, for every $\cH \in V(H)$, $d_H(\cH) = M \paren{\binom{n}{k}-M}$.
Now define a random walk $m$ on $H$ as follows:
Let $\cH \in V(H)$. Define 
\[m_{\cH}(\cH') = \begin{cases}
 \frac{1}{M \paren{\binom{n}{k}-M}+1}  & \textrm{ if $\cH' \in \Gamma(\cH)$,}  \\
    0      &  \textrm{ otherwise.}
\end{cases}\]
By the same logic in Section \ref{sec: randomMedges}, we can obtain a lower bound for the Ricci curvature of $H$, i.e., for all $\cH_1, \cH_2 \in V(H)$,
$$\kappa(\cH_1, \cH_2) \geq \frac{\binom{n}{k}}{M\paren{\binom{n}{k}-M}+1}.$$

Similar to before, we can also apply Theorem \ref{thm:main_tool} to obtain concentration results for the number of copies of fixed sub-hypergraphs in a uniformly random hypergraph on $n$ vertices and $M$ edges. The idea is similar to Section \ref{sec: randomMedges} and we leave the details to the readers.

\subsection{Vertex-Lipschitz functions on random \texorpdfstring{$d$}{d}-out(in)-regular graphs}
Given a directed graph $G$ and a vertex $v$, we use $\delta^{+}(v)$ and $\delta^{-}(v)$ to denote the outdegree and indegree, respectively, of a vertex $v$.
A \textit{$d$-out-regular graph} $G$ is a directed graph in which $\delta^{+}(v) = d$ for every $v\in V(G)$. Similarly, a \textit{$d$-in-regular graph} $G$ is a directed graph in which $\delta^{-}(v) = d$ for every $v\in V(G)$. Moreover, let $\Gamma^+(v)= \{u \in V(G): vu \in E(G)\}$, $\Gamma^{-}(v)= \{u \in V(G): uv \in E(G)\}$, $N^+(v) = \Gamma^+(v) \cup \{v\}$ and $N^-(v) = \Gamma^-(v) \cup \{v\}$.

Let $H$ be a graph such that $V(H)$ is the set of all labeled $d$-out-regular graphs on $n$ vertices. Two graphs $G_1, G_2 \in V(H)$ are adjacent in $H$ if and only if there exists some vertex $v \in V(G_1) = V(G_2)$ such that one can be obtained from the other by changing $\Gamma^+(v)$. It is not hard to see that $H$ is a connected graph with $diam(H) \leq n$. Moreover, it is also clear that if $G_1, G_2$ are adjacent in $H$, there is a unique vertex $v$ such that one can be obtained from the other by changing $\Gamma^+(v)$.

Now define a random walk $m$ on $H$ as follows: let $G \in V(H)$ and define

\[m_G(G') = \begin{cases}
 \frac{1}{n \paren{\binom{n-1}{d}-1}+1}  & \textrm{ if $G'\in N^+(G)$,} \\
    0      &  \textrm{ otherwise}.
\end{cases}\]
It's easy to see that for a fixed $G$, $\sum_{G'} m_G(G') = 1$.

\begin{proposition}
Let $\nu$ be the unique invariant distribution of the random walk defined above. A random graph $G$ picked according to $\nu$, is equally likely to be one of the $d$-out-regular graphs on $n$ vertices.
\end{proposition}
\begin{proof}
Observe that $H$ is not bipartite thus the random walk is ergodic.
There are $\binom{n-1}{d}^n$ many $d$-out-regular graphs in total. Hence, it suffices to show that $\nu'(G) = \binom{n-1}{d}^{-n}$ for every $G$ is an invariant distribution for the random walk. Indeed, for every fixed $G \in V(H)$,
\begin{align*}
    \dss_{G'\in H} \nu'(G') m_{G'}(G) &= \binom{n-1}{d}^{-n} \dss_{G' \in H} m_{G'}(G)\\
                &= \binom{n-1}{d}^{-n} \dss_{G' \in H} m_{G}(G')\\
                &= \binom{n-1}{d}^{-n}\\
                &= \nu'(G).
\end{align*}
Since $\nu$ is the unique invariant distribution, it follows then that $\nu= \nu'$. 
\end{proof}

\begin{lemma}
Let $H$ and the random walk $m$ be defined as above. Then $$\kappa(G_1, G_2) \geq \frac{1}{n}$$ for all $G_1, G_2 \in V(H)$.
\begin{proof}
Again, by Lemma \ref{lem:kappa_neighbor}, we can assume that $G_1, G_2$ are neighbors in $H$. It then follows from definition that
\[\kappa(G_1, G_2) = 1- W(m_{G_1},m_{G_2}).\] 

Suppose $v$ is the unique vertex such that $G_2$ can be obtained from $G_1$ by changing $\Gamma^+(v)$.
Consider the support of $m_{G_1}$. For each $G_1' \in N(G_1)$, we will match $G_1'$ with a distinct graph $\phi(G_1') \in N(G_2)$. Again, let $\phi(G_1) = G_1$ and $\phi(G_2) = G_2$. For other neighbors $G_1'$ of $G_1$, there are two possible cases:
\begin{description}
\item Case $1$: $G_1 - v = G_1' -v$. Then it follows that $G_1'-v = G_2 -v$ and we let $\phi(G_1') = G_1'$.
\item Case $2$: $G_1 -u = G_1'- u$ for some $u\neq v$. In this case, we claim that for each $G_1'$ such that $G_1 -u = G_1' -u$, there exists a unique $G_2' = \phi(G_1')$ such that $G_2' - u = G_2 - u$ and $G_1' - v = G_2' - v$.
Indeed, let $G_2'$ be obtained from $G_2$ by replacing the out-neighbors of $u$ in $G_2$ by the out-neighbors of $u$ in $G_1'$. It's not hard to see that $G_2'-u = G_2 -u$ and $G_1' - v = G_2'-v$.
\end{description}

Let us now define a coupling $A$ (not necessarily optimal) between $m_{G_1}$ and $m_{G_2}$. Define $A: V(H) \times V(H) \to \mathbb{R}$ as follows:
\begin{equation}
    A(G_1', G_2') = \begin{cases}
      \frac{1}{n \paren{\binom{n-1}{d}-1}+1} & \textrm{ if } G_1' \in N(G_1)\textrm{ and $G_2' = \phi(G_1')$,}\\
        0    &\textrm{ otherwise. }
    \end{cases}
\end{equation}

It is not hard to verify that $A$ is a coupling of $m_{G_1}$ and $m_{G_2}$. 
Now by definition,
\begin{align*}
    W(m_{G_1},m_{G_2}) &\leq \dss_{G_1',G_2'} A(G_1', G_2') d(G_1', G_2') \\
    &\leq \dss_{u\neq v}\dss_{\substack{G_1' \in N(G_1)\\G_1'-u = G_1-u}} A(G_1', \phi(G_1')) d(G_1', \phi(G_1'))\\
    &\leq (n-1) \paren{\binom{n-1}{d}-1} \frac{1}{n \paren{\binom{n-1}{d}-1}+1}
\end{align*}
It follows that 
    \[\kappa(G_1, G_2) = 1- W(m_{G_1},m_{G_2}) \geq \frac{\binom{n-1}{d}}{n\lp  \binom{n-1}{d}-1\rp+1} \geq \frac{1}{n}.\]
This completes the proof of the lemma.
\end{proof}
\end{lemma}

Let $G$ be a uniformly random $d$-out-regular graph. A \textit{directed triangle} is a cycle of length $3$ with vertices $u,v,w$ such that $uv, vw$ and $wu$ are all directed edges. Let $X_{n,d}:= X(G)$ be the random variable denoting the number of directed triangle in $G$. It is not hard to see that 

$$E[X_{n,d}] \approx 2\binom{n}{3} \paren{\frac{d}{n-1}}^3.$$
We will now use Theorem \ref{thm:main_tool} to derive the concentration behavior of $X_{n,d}$. Note that $X_{n,d}$ is $(d^2)$-Lipschitz. Hence by Theorem \ref{thm:main_tool}, we have that

$$ \Prob{\abs*{\frac{X_{n,d}}{d^2} - \frac{E[X_{n,d}]}{d^2}} > \frac{t}{d^2}} \leq 2\exp{\lp -\frac{t^2\kappa}{5d^4}\rp}.$$ 
It follows that 

$$\Prob{\abs*{X_{n,d} - E[X_{n,d}]} > t} 
         \leq 2\exp{\lp -\frac{t^2\kappa}{5d^4}\rp}
         \leq 2\exp{\lp -\frac{t^2}{5nd^4}\rp}.$$

\subsection{Lipschitz functions on random linear permutations}

We will denote a linear permutation $\sigma$ by $\sigma = [a_1 a_2 \ldots a_n]$ such that $a_i \in [n]$ for all $i$ and $\sigma(i) = a_i$. A linear permutation on $[n]$ can be viewed as a sequence of $n$ distinct numbers from $[n]$. Thus, WLOG, $\{a_1, a_2, \ldots, a_n\} = [n]$.
Given two permutations $\sigma_1, \sigma_2$ where $\sigma_1 = [a_1 a_2 \ldots a_n]$, we say $\sigma_1$ is \textit{$(i,j)$-alike} to $\sigma_2$ if $\sigma_2$ can be obtained from $\sigma_1$ by moving the number $i$ to the position after the number $j$ in $\sigma_1$; moreover, $\sigma_1$ is $(i,0)$-alike to $\sigma_2$ if $\sigma_2$ can be obtained from $\sigma_1$ by moving the number $i$ to the first position of $\sigma_1$. For example, $\sigma_1 = [12345]$ is $(2,4)$-alike to $\sigma_2 = [13425]$ and is $(4,0)$-alike to $\sigma_3 = [41235]$.
Two \textit{distinct} linear permutations $\sigma_1, \sigma_2$ are \textit{insertion-alike} if one is $(i,j)$-alike to the other for some $i\neq j$.

Let $H$ be the graph such that $V(H)$ is the set of all linear permutations of $[n]$ and two linear permutation $\sigma_1, \sigma_2$ are adjacent in $H$ if and only if they are insertion-alike. Clearly $H$ is a connected graph with diameter at most $n$. Moreover, every vertex (which is a linear permutation) in $H$ has $(n-1)^2$ neighbors in $H$.

Now define a random walk $m_{\alpha}$ on $H$ as follows: let $\sigma \in V(H)$ and define
\[m_{\sigma}(\sigma') = \begin{cases}
     \frac{1}{(n-1)^2+1}  & \textrm{ if $\sigma = \sigma'$ or $\sigma$ is insertion-alike to $\sigma'$,}\\
    0   &  \textrm{ otherwise}.
\end{cases}\]
It's not hard to see that for a fixed $\sigma$, $\sum_{\sigma'} m_{\sigma}(\sigma') = 1$. Moreover, $m_{\sigma}(\sigma') = m_{\sigma'}(\sigma)$ for every pair of $\sigma, \sigma'$.

\begin{proposition}
Let $\nu$ be the unique invariant distribution of the random walk defined above. A random permutations $\sigma$ picked according to $\nu$, is equally likely to be one of the $n!$ permutations.
\end{proposition}
\begin{proof}
Observe that $H$ is not bipartite thus the random walk is ergodic.
There are $n!$ permutations in total. Hence, it suffices to show that $\nu'(\sigma) = (n!)^{-1}$ for every $\sigma$ is an invariant distribution for the random walk. 
\begin{align*}
    \dss_{\sigma'\in H} \nu'(\sigma') m_{\sigma'}(\sigma) &= \frac{1}{n!} \dss_{\sigma' \in V(H)} m_{\sigma'}(\sigma)\\
                &= \frac{1}{n!} \dss_{\sigma' \in V(H)}     m_{\sigma}(\sigma')\\
                &= \frac{1}{n!}\\
                &= \nu'(\sigma).
\end{align*}
Since $\nu$ is the unique invariant distribution, it follows then that $\nu= \nu'$. 
\end{proof}

\begin{lemma}
Let $H$ and the random walk $m$ be defined as above.
If $\sigma_1, \sigma_2 \in V(H)$ are neighbors in $H$, then $\kappa(\sigma_1, \sigma_2) \geq \frac{1}{n}$.
\end{lemma}

\begin{proof}
WLOG, suppose that $\sigma_1$ is $(i,j)$-alike to $\sigma_2$ (with $\sigma_2 \neq \sigma_1$).
Consider the support of $m_{\sigma_1}$. For each $\sigma_1' \in N(\sigma_1)$, we will match $\sigma_1'$ with a distinct permutation $\phi(\sigma_1') \in N(\sigma_2)$. 
First let $\phi(\sigma_1) = \sigma_1$ and $\phi(\sigma_2) = \sigma_2$. For other neighbors $\sigma_1'$ of $\sigma_1$, there are two cases:
\begin{description}
\item Case $1$: $\sigma_1$ is $(i,k)$-alike to $\sigma_1'$ where $k\neq j$. Then it follows that $\sigma_1'$ is also $(i,j)$-alike to $\sigma_2$ and we let $\phi(\sigma_1') = \sigma_1'$.
\item Case $2$: $\sigma_1$ is $(i', j')$-alike to $\sigma_1'$ where $i' \neq i$ and $\sigma_1$ is not $(i,k)$-alike to $\sigma_1'$ for any $k$. In this case, let $\sigma_2'$ be the permutation such that $\sigma_2$ is $(i',j')$-alike to $\sigma_2'$. It follows easily that $\sigma_1'$ is also $(i,j)$-alike to $\sigma_2'$. We then define $\phi(\sigma_1') = \sigma_2'$.
\end{description}
Let us now define a coupling $A$ (not necessarily optimal) between $m_{\sigma_1}$ and $m_{\sigma_2}$. Define $A: V(H) \times V(H) \to \mathbb{R}$ as follows:
\begin{equation}
    A(\sigma_1', \sigma_2') = \begin{cases}
       \frac{1}{(n-1)^2+1} & \textrm{ if } \sigma_1' \in N(\sigma_1) \textrm{ and $\sigma_2' = \phi(\sigma_1')$,}\\
        0    &\textrm{ otherwise. }
    \end{cases}
\end{equation}
It is not hard to verify that $A$ is a coupling of $m_{\sigma_1}$ and $m_{\sigma_2}$. 
Now by definition,
\begin{align*}
    W(m_{\sigma_1},m_{\sigma_2}) &\leq \dss_{\sigma_1',\sigma_2'} A(\sigma_1', \sigma_2') d(\sigma_1', \sigma_2') \\
    &\leq \dss_{\sigma' \in N(\sigma_1)} A(\sigma_1', \phi(\sigma_1')) d(\sigma_1', \phi(\sigma_1'))\\
    &\leq 1 - \frac{n}{(n-1)^2+1}.
\end{align*}
It follows that 
    \[\kappa(\sigma_1, \sigma_2) = 1- W(m_{\sigma_1},m_{\sigma_2}) \geq \frac{n}{(n-1)^2+1} \geq \frac{1}{n}.\]
This completes the proof of the lemma.
\end{proof}

Now we give an example of concentration results on the space of random linear permutations. In particular, we discuss the number of occurrences of certain patterns in random permutations. Denote the set of length $n$ linear permutations by $\mathcal{S}_n$. Given a permutation pattern $\tau \in \mathcal{S}_k$, we say that a permutation $\pi = [\pi_1 \ldots \pi_n] \in \mathcal{S}_n$ contains the pattern $\tau$ if there exists $1\leq i_1 < i_2 < \ldots < i_k \leq n$ such that the $\pi_{i_s} < \pi_{i_t}$ if and only if $\tau_s < \tau_t$ for every pair $s,t$.
Each such subsequence in $\pi$ is called an \textit{occurrence} of the pattern $\tau$. Let $\tau$ be a random permutation in $\mathcal{S}_n$ and let the random variable $X_{\tau, n}:= X_{\tau}(\pi)$ be the number of copies of $\tau$ in $\pi$.
We consider asymptotics as $n\to \infty$ for (one or several) fixed $\tau$.

The (joint) distribution of the $X_{\tau, n}$ has been investigated in a series of paper \cite{Bona, Bona2, JNZ}. In particular, Bona \cite{Bona} showed that for every $\tau \in \mathcal{S}_k$, as $n\to \infty$, 
\begin{equation}\label{eq:normal}
    \frac{X_{\tau, n}- E[X_{\tau, n}]}{n^{k-\frac{1}{2}}} \to N(0, Z_{\tau})
\end{equation}
for some $Z_{\tau} > 0$.
Janson, Nakamura and Zeilberger \cite{JNZ} showed that the above holds jointly for any finite family of patterns $\tau$.

Note that as a consequence of the convergence in \eqref{eq:normal}, we obtain the following concentration inequality:
\begin{equation}
\Prob{\abs*{X_{\tau,n} - E[X_{\tau,n}]} > t \rp \leq 2\exp \lp - \frac{t^2}{2n^{2k-1} Z_{\tau}}}
\end{equation}
which is sharp up to a polynomial factor.

On the other hand, consider the graph $H$ defined at the beginning of this subsection, where $V(H)$ is the set of all linear permutations of $[n]$. 
It is not hard to see that the function $X_{\tau, n}: V(H) \to \mathbb{Z}$ is $\binom{n-1}{k-1}$-Lipschitz. It follows by Theorem \ref{thm:main_tool} that 
\begin{align*}
    \Prob{\abs*{\frac{X_{\tau,n}}{\binom{n-1}{k-1}} -   \frac{E[X_{\tau,n}]}{\binom{n-1}{k-1}}} > \frac{t}{\binom{n-1}{k-1}}} 
        & \leq 2\exp \lp -\frac{t^2 \kappa}{5\binom{n-1}{k-1}^2}  \rp\\
        & \leq 2\exp \lp -\frac{t^2}{C_k n^{2k-1}} \rp
\end{align*}
for some $C_k > 0$. Hence the concentration result in Theorem \ref{thm:main_tool} is in fact asymptotically optimal in the case of counting occurrences of patterns in random permutations.

\begin{remark}
Similar Ricci curvature and concentration results can be obtained for the space of cyclic permutations as well.
\end{remark}

\begin{remark}
Another possible way to geometrize the space of linear permutations is the random transposition model (see, e.g., \cite{EMT}) as follows: let $V(H) = \mathcal{S}_n$ and two permutations $\sigma_1, \sigma_2 \in V(H)$ are adjacent in $H$ if $\sigma_2 = \tau \circ \sigma_1$ for some transposition $\tau$. Define a random walk $m$ on $H$ by 
$$m_{\sigma}(\sigma') = \begin{cases}
     \frac{2}{n(n-1)}  & \textrm{ if $\sigma$ and $\sigma'$ are adjacent in $H$,}\\
    0   &  \textrm{ otherwise}.\end{cases}$$
\end{remark}
The invariant distribution is the uniform measure on $\mathcal{S}_n$. The Ricci curvature of this graph is $\Theta(n^{-2})$, as observed by Gozlan et al \cite{GMPRST}.

\end{document}